\documentclass[leqno,a4paper]{article}
\usepackage{amsmath,amsthm,amssymb}
\usepackage[T1]{fontenc}

\newtheorem{Prop}{Proposition}
\newtheorem{Lemma}{Lemma}
\newtheorem{Remark}{Remark}

\newtheorem{Cor}{Corollary}

\newcommand{\beq}{\begin{equation}}
\newcommand{\eeq}{\end{equation}}

\def\scalar(#1,#2){(#1\mid#2)}

\newcommand{\cd}{{\cal D}}

\newcommand{\ch}{{\cal H}}

\newcommand{\co}{{\cal O}}

\newcommand{\ycn}{(Y,{\cal C},\nu)}

\newcommand{\ov}{\overline}

\newcommand{\bs}{\mathbb{S}}

\newcommand{\R}{{\mathbb{R}}}
\newcommand{\T}{{\mathbb{T}}}
\newcommand{\C}{{\mathbb{C}}}
\newcommand{\Z}{{\mathbb{Z}}}
\newcommand{\N}{{\mathbb{N}}}

\newcommand{\va}{\varphi}

\newcommand{\mob}{\boldsymbol{\mu}}
\newcommand{\tend}[3][]{\xrightarrow[#2\to#3]{#1}}

\title{0-1 sequences of the Thue-Morse type and Sarnak's conjecture}
\author{El Houcein El Abdalaoui\and Stanis\l aw Kasjan \and Mariusz Lema\'nczyk\thanks{Research supported by Narodowe Centrum Nauki grant DEC-2011/03/B/ST1/00407}}

\begin{document}

\maketitle \normalsize

\thispagestyle{empty}

\begin{abstract}
We show that the images via $z\mapsto z^m$
of the continuous part of the spectral measures of the dynamical systems generated by the 0-1 sequences of the Thue-Morse type  are pairwise mutually singular  for different odd numbers $m\in\N$. Sarnak's conjecture on orthogonality with the
M\"obius function is shown to hold for such dynamical
systems.  The same conjecture is shown to hold for all systems induced by regular Toeplitz sequences. A non-regular Toeplitz sequence for which Sarnak's conjecture fails is constructed.
\end{abstract}

\section{Introduction} In this article we will deal with dynamical systems generated by generalized Morse sequences introduced to ergodic theory  by M.
Keane \cite{Ke} in 1968. Recall that a sequence $x\in\{0,1\}^{\N}$ is called a {\em generalized Morse sequence} if
\beq\label{defm}
x=b^0\times b^1\times \ldots\eeq
with $b^i\in\{0,1\}^{p_i}$, $p_i\geq2$, $b^i(0)=0$ for each $i\geq0$, where,
given blocks $B\in\{0,1\}^k$ and $C=C(0)C(1)\ldots C(\ell-1)\in\{0,1\}^{\ell}$, we set
$$
B\times C=B^{C(0)}B^{C(1)}\ldots B^{C(\ell-1)}$$
with $B^0=B$ and $B^1$ arises from $B$ by the interchange of $0$s and $1$s.

Consider now $X=\co(x)\subset\{0,1\}^{\Z}$ the subset of all two-sided sequences such that each block of consecutive symbols appearing in $y\in \co(x)$ also appears in $x$. Denote by $T:\{0,1\}^{\Z}\to\{0,1\}^{\Z}$ the shift, i.e.\ the homeomorphism that shifts a two-sided sequence of 0s and 1s  by one position to the left. Then $\co(x)$ is  closed and $T$-invariant. Under some mild assumptions \cite{Ke} on the blocks $b^0,b^1,\ldots$, one obtains a strictly ergodic dynamical system $(X,\mu_x,T)$, where $\mu_x$ is a unique $T$-invariant (Borel) probability measure (it is given by the average frequencies of blocks on $x$). From now on, only strictly ergodic case is considered.  ~\footnote{There are $y\neq z\in X$ such that $y(i)=z(i)=x(i)$ for each $i\geq 0$, therefore, the systems considered in this article are not distal, cf.\ \cite{Li-Sa}.}

We will consider
Sarnak's conjecture \cite{Sa} which
says that whenever $S$ is a zero (topological) entropy  homeomorphism of a compact metric space $Y$ then, for each $f\in C(Y)$ and $y\in Y$, we have
\beq\label{s1}\frac1N\sum_{n=1}^Nf(S^ny)\mob(n)\to0~\footnote{Following \cite{Sa}, when~(\ref{s1}) holds, we speak about the orthogonality of the sequence $(f(S^ny))$ with the M\"obius function.}, \eeq where
$\mob(\cdot)$ stands for the classical M\"obius function\footnote{$\mob(1)=1$, $\mob(n)=0$ for any non-square free number $n$, and $\mob(p_1p_2\ldots p_k)=(-1)^k$ for $p_1,p_2,\ldots,p_k$ different prime numbers.}. In numerous cases, this conjecture has recently been proved to hold: \cite{Ab-Le-Ru}, \cite{Bo}, \cite{Bo-Sa-Zi}, \cite{Gr-Ta}, \cite{Ku-Le}, \cite{Li-Sa}, \cite{Ma-Ri2} (see also \cite{Ta1}).


Consider first the simplest subclass of the class of generalized Morse sequences, namely, those sequences~(\ref{defm}) for which $|b^i|=2$ for all $i\geq0$ (in other words, either $b^i=01$ or $b^i=00$). Such sequences are called Kakutani sequences \cite{Kw}. A particular case of Sarnak's conjecture, namely:
\beq\label{kk}
\frac1N\sum_{n=1}^N(-1)^{x(n)}\mob(n)\to0,\eeq
for the classical Thue-Morse sequence $x=01\times01\times\ldots$ follows from \cite{In-Ka} and \cite{Ka} (see also \cite{Da-Te} where, additionally, the speed of convergence to zero is given and \cite{Ma-Ri1}, where, additionally, a PNT has been proved).
Then~(\ref{kk}) has been proved for some subclass of Kakutani sequences in \cite{Gr}. As a matter of fact, in~\cite{Gr},   the problem whether $\frac1N\sum_{n=1}^N(-1)^{s_E(n)}\mob(n)\to 0$ is considered. Here
$E\subset\N$ is fixed and $s_E(n):=\sum_{i\in E}n_i$, where $n=\sum_{i=0}^\infty n_i2^i$ ($n_i\in\{0,1\}$). To see a relationship with Kakutani sequences define a Kakutani sequence $x=b^0\times
b^1\times\ldots$ with $b^n=01$ iff $n+1\in E$; it is now not hard to see that $s_E(n)=x(n)$ mod~2 (for each $j\in\{0,1,\ldots,2^n-1\}$ consider its binary expansion which, by putting some zeros on the left if necessary, is of length~$n$; then the binary expansion of $j+2^n$ has length $n+1$ and begins by~$1$; it follows that $s_E(j+2^n)\neq s_E(j)$ modulo~$2$ (for all $j$ as above) iff $n+1\in E$).  Finally, using some methods from~\cite{Ma-Ri1}, Bourgain \cite{Bo0} completed the result from \cite{Gr} so that~(\ref{kk}) holds in the whole class of Kakutani sequences (moreover, in \cite{Bo0}, \cite{Gr}  a relevant PNT has been proved).
As a matter of fact, the methods used in the aforementioned papers allow us to have~(\ref{kk}) with $x$ replaced by every $y\in\co(x)$.

Coming back to dynamical systems induced by generalized Morse sequences, two natural problems can now be formulated: to deal with the full version of Sarnak's conjecture~(\ref{s1}) and to deal with blocks $b^i$ of arbitrary length. From this point of view, the present article can be seen as first steps to handle both these problems. Borrowing an idea from \cite{Ab-Le-Ru}, we consider Sarnak's conjecture~(\ref{s1}) in the class of dynamical systems given by those generalized Morse sequences that have the classical Thue-Morse sequence
as a stabilizing subsequence (in what follows, we call such sequences of being of Thue-Morse type). More precisely, we say that a sequence~(\ref{defm}) is {\em of Thue-Morse type} if for each $m\geq1$ there is $i_m$ such that
$$b^{i_m}=b^{i_m+1}=\ldots= b^{i_m+m-1}=01,$$ that is
$$
\begin{array}{lll}x&=&b^0\times\ldots\times b^{i_1-1}\times 01\times b^{i_1+1}\times\ldots \times b^{i_2-1}\times 01\times 01\times b^{i_2+2}\times\ldots\times\\ &&b^{i_3-1}\times01\times01\times01\times b^{i_3+3}\times\ldots\end{array}$$
with the blocks $b^r$ above of arbitrary length ($b^r(0)=0$). Notice that the intersection of the class of Thue-Morse type sequences with the class of Kakutani sequences corresponds to the sets $E\subset\N$  that contain intervals of consecutive integers of arbitrary lengths.

We now pass to a description of the content of the article. We  deal
with the problem of the mutual singularity of the images via the maps $\bs^1\ni z\mapsto z^m\in \bs^1$, $m\in\N$, of the continuous part of the spectral type of the dynamical systems $(\co(x),T,\mu_x)$ for a sequence~(\ref{defm}) of Thue-Morse type.  Although the spectral theory of
Morse dynamical systems is rather well-known \cite{Gu}, \cite{Kw}, the spectral problem formulated above does not seem to be taken up yet in the literature and seems to be of independent interest.
One of our main results  shows that if $\widetilde{\sigma}$ denotes the continuous part of the spectral type  of the dynamical system $(\co(x),T,\mu_x)$ arising from a Thue-Morse type sequence $x$, then $\widetilde\sigma^{(r)}\perp\widetilde\sigma^{(s)}$ whenever neither $r$ is a $2^k$-multiple (for some $k\geq0$) of $s$, nor vice versa; here, for $m\in\N$, $\widetilde\sigma^{(m)}$ stands for the image of $\widetilde\sigma$ via $z\mapsto z^m$. In particular, we have the mutual singularity of measures $\widetilde\sigma^{(r)}$ and $\widetilde\sigma^{(s)}$ whenever $r,s$ are two different odd numbers.

The idea to use Furstenberg's disjointness \cite{Gl} to cope with Sarnak's conjecture for a uniquely ergodic system $(X,T)$ has already appeared in \cite{Bo-Sa-Zi}. However, in some cases this idea  cannot  be applied directly (for example, if the systems under consideration are not weakly mixing, which is the case of the present article), the spectral approach does help, see recent  \cite{Ab-Le-Ru}, \cite{Bo}. This approach -- a mixture of spectral disjointness with a direct proof of the M\"obius orthogonality of the sequences corresponding to the discrete part of the spectrum --  will also be applied here.

The plan of the article is, firstly, to give a complete proof of the mutual spectral singularity result for the spectral measure of the classical Thue-Morse dynamical system $(X,T,\mu_x)$ itself via some analysis of the Fourier transform of the spectral measure and using a simple general singularity criterion. This will allow us to  reduce the problem of the validity of Sarnak's conjecture for $(X,T,\mu_x)$ to the same problem for some natural topological factor. This factor is determined by so called  Thue-Toeplitz sequence~\footnote{This factor is a particular uniquely ergodic model of the dyadic odometer.}. In order to complete the proof, we show that Sarnak's conjecture is true for all dynamical systems given by so called  regular Toeplitz sequences \cite{Ja-Ke}, \cite{Wi}, see Section~\ref{ostatnia}. This general result, in Section~\ref{ostatnia1}, enables us to conclude the validity of Sarnak's conjecture for all dynamical systems arising from the Thue-Morse type $0-1$-sequences. On the other hand, we construct a non-regular Toeplitz sequence (the entropy of it is positive but is not computed here) for which Sarnak's conjecture fails.

The reader should however be aware that our spectral approach to prove Sarnak's conjecture in the class of Morse dynamical systems may be insufficient. In view of a result of M.\ Guenais~\cite{Gu},  the problem whether there are generalized Morse sequences for which the continuous part of the spectral measure of the corresponding dynamical system  is Lebesgue is equivalent to the open classical problem of the existence of flat trigonometric polynomials with coefficients $\pm1$, i.e.\ we seek trigonometric polynomials $P_n(t)=\sum_{j=0}^{p_n-1}\epsilon_{j,n}e^{2\pi ijt}$ ($\epsilon_{j,n}=\pm1$) such that $\lim_{n\to\infty}\|P_n\|_{L^1}/\sqrt{p_n}=1$.

For necessary notions and facts from ergodic theory and spectral theory of dynamical systems we refer the reader to e.g.\ \cite{Gl}, \cite{Ka-Th}. For the theory of substitutions, also from the dynamical system point of view, see \cite{Qu}.

The authors would like to thank Christian Mauduit for fruitful discussions on the subject.

\section{Weighted operators and disjointness of spectral measures} Assume that $S$ is an ergodic automorphism of a standard probability space $\ycn$. Let $\varphi:Y\to\Z_2:=\{0,1\}$ be a measurable map (cocycle).  Denote by $W_{S,\va}$ the corresponding {\em weighted operator} on $L^2\ycn$:
\beq\label{wo1}
W_{S,\va}(\xi)(y)=(-1)^{\va(y)}\xi(Sy)\;\;\;\mbox{for}\;\;\xi\in L^2\ycn.\eeq
Then, for each $m\in\Z$
$$
\left(W_{S,\va}^m(\xi)\right)(y)=(-1)^{\va^{(m)}(y)}\xi(S^my),$$
where $\va^{(m)}(y)=\va(y)+\va(Sy)+\ldots+\va(S^{m-1}y)$ for $m\in
\Z$ positive and
$\va^{(m_1+m_2)}(y)=\va^{(m_1)}(y)+\va^{(m_2)}(S^{m_1}y)$ for each
$m_1,m_2\in\Z$ ($\va^{(0)}=0$).

The following observation is folklore (we provide a proof for completeness).

\begin{Prop}\label{pwo} Assume additionally that $S$ has discrete spectrum.  Assume moreover that $S^{q_n}\to Id$ on $L^2\ycn$ for some increasing sequence $q_n\to\infty$ of natural numbers. If for some $c\in\R$, $|c|\leq1$
\beq\label{wo2}
\int_Y(-1)^{\va^{(q_n)}}\,d\nu\to c\eeq
then $\left(W_{S,\va}\right)^{q_n}\to c\cdot Id$ weakly on $L^2\ycn$.
\end{Prop}
\begin{proof} By assumption, $L^2\ycn$ is generated by eigenfunctions
$\eta_i\in L^2\ycn$ ($|\eta_i|=1$), $\eta_i\circ S=\alpha_i\cdot \eta_i$ for
 some $\alpha_i\in\C$ ($|\alpha_i|=1$), $i\geq1$. Now, for $i,j\geq1$ we have
$$
\big\langle\left(W_{S,\va}\right)^{q_n}\eta_i,\eta_j\big\rangle=\int_Y(-1)^
{\va^{(q_n)}}\cdot\eta_i\circ S^{q_n}\cdot\ov{\eta}_j\,d\nu$$
$$=\alpha_i^{q_n}\int_Y
(-1)^{\va^{(q_n)}}\cdot\eta_i\cdot\ov{\eta}_j\,d\nu\rightarrow
c\langle\eta_i,\eta_j\rangle$$ since $\alpha_i^{q_n}\to 1$ and for
$i\neq j$ we have
$$\int_Y (-1)^{\va^{(q_n)}}\cdot\eta_i\cdot\ov{\eta}_j\,d\nu=\int_Y (-1)^{\va^{(q_n)}\circ S}\cdot(\eta_i\cdot\ov{\eta}_j)\circ S\,d\nu$$
$$
=\alpha_i\cdot \ov{\alpha}_j\int_Y (-1)^{\va^{(q_n)}\circ S-\va^{(q_n)}} (-1)^{\va^{(q_n)}}\cdot\eta_i\cdot\ov{\eta}_j\,d\nu$$$$=
\alpha_i\cdot \ov{\alpha}_j\int_Y (-1)^{\va\circ S^{q_n}-\va} (-1)^{\va^{(q_n)}}\cdot\eta_i\cdot\ov{\eta}_j\,d\nu,$$
so $\int_Y (-1)^{\va^{(q_n)}}\cdot\eta_i\cdot\ov{\eta}_j\,d\nu\to 0$ as (by the ergodicity of $S$) $\alpha_i\neq\alpha_j$ and $(-1)^{\va\circ S^{q_n}-\va}\to 1$ in measure.\end{proof}

Given a measure $\sigma$ on $\bs^1$ denote by $\sigma^{(k)}$ the image of $\sigma$ via the map $z\mapsto z^k$.

Consider now $\xi={\mathbf 1}={\mathbf1}_Y$ and note that the Fourier transform of the spectral measure\footnote{Given a unitary operator $U$ of a Hilbert space $\ch$ and $\xi\in\ch$, the spectral measure $\sigma_\xi$ is determined by
$
\widehat{\sigma}_{\xi}(m):=\int_{\bs^1}z^m\,d\sigma_{\xi}(z)=\langle U^m\xi,\xi\rangle$ for all $m\in\Z$.} $\sigma_{\mathbf1}$ of $\mathbf1$ is given by
$$\widehat{\sigma}_{\mathbf 1}(m)=\int_Y(-1)^{\va^{(m)}}\,d\nu,\;\;m\in\Z.$$
Since weak convergence in Proposition~\ref{pwo} is preserved when passing to closed, invariant subspaces, we obtain the following result.

\begin{Cor}\label{cwo} Under the assumptions of Proposition~\ref{pwo}, if $r\neq s$ and for some $t\in\Z$ we have
\beq\label{wo3}
 \int_Y(-1)^{\va^{(rtq_n)}}\,d\nu\to c_1,\eeq
 \beq\label{wo4}
 \int_Y(-1)^{\va^{(stq_n)}}\,d\nu\to c_2\eeq
 with $c_1\neq c_2$, then
 \beq\label{wo5}
 \sigma^{(r)}_{\mathbf1}\perp\sigma^{(s)}_{\mathbf1}.\eeq
 \end{Cor}

 \section{Generalized Morse sequences and strongly $p$-multiplicative sequences}

Let
$$
x=b^0\times b^1\times \ldots
$$
be a generalized Morse sequence. Denote by $(\co(x),T,\mu_x)$ the corresponding Morse dynamical system.
Consider the function $f:\co(x)\to\R$,
\beq\label{ms1}
f(y)=(-1)^{y(0)},\;\;y\in \co(x).\eeq
Then $f$ is continuous, and the spectral measure $\sigma_f$ of $f$ under $T$ is determined by (see \cite{Kw}, \cite{Le})  \beq\label{ms2}
\widehat\sigma_f(m)=\int_{\co(x)}(-1)^{y(m)-y(0)}\,d\mu_x(y),\;\;\;m\in\Z.\eeq
Moreover, the dynamical system $(\co(x),\mu_x,T)$ has a representation as a skew product $S_\varphi$
over $S$ which is the $(q_n)$-odometer ($q_n=p_0\cdot\ldots\cdot p_{n-1}$, $S$ is defined on $\ycn$) and a cocycle
$\va:Y\to\Z_2$ \cite{Ke}, \cite{Le} (called in \cite{Le} a Morse cocycle). From the spectral point of view, in this representation, the function $f$ (given by~(\ref{ms1})) corresponds to the function ${\mathbf1}_Y$ for the weighted operator $W_{S,\va}$, in other words (see \cite{Le})
\beq\label{ms3}
\widehat{\sigma}_f(m)=\int_Y(-1)^{\va^{(m)}(y)}\,d\nu(y),\;\;m\in\Z.\eeq

A special case arises when $b^0=b^1=\ldots=B$. Then the corresponding (generalized) Morse sequence $x$ can be obtained as a fixed point for the following substitution (see \cite{Qu})
$$
0\mapsto B,\;\;1\mapsto \widetilde{B}.$$ In this particular case,
the sequence $m_n(x):=(-1)^{x(n)}$, where $x=B\times
B\times\ldots$, is {\em strongly $p$-multiplicative}, where $p$ is
the length of $B$, that is \beq\label{ms4}
m_{ap^n+b}(x)=m_a(x)\cdot m_b(x)\;\;\mbox{for
each}\;a,b\geq0,\;b<p^n.\eeq Each $p$-multiplicative sequence
$(m_n(x))$  determines its {\em spectral measure} $\sigma_x$
\cite{Qu} whose Fourier transform $\widehat{\sigma}(\cdot)$ is
given by \beq\label{ms5}
\widehat\sigma_x(k)=\lim_{N\to\infty}\frac1N\sum_{n=1}^N
m_{n+k}(x)\cdot m_n(x),\;\;k\in\Z.\eeq Since (by~(\ref{ms4}))
$$
\widehat\sigma_x(kp)=\lim_{N\to\infty}\frac1N\sum_{n=1}^Nm_{n+kp}(x)\cdot m_n(x)=$$
$$=\lim_{N\to\infty}\frac1N\sum_{n=1}^{[N/p]}\sum_{i=0}^{p-1}m_{pn+i+kp}(x)\cdot m_{pn+i}(x)=
$$$$
\lim_{N\to\infty}\frac1{[N/p]}\sum_{n=1}^{[N/p]}m_{n+k}(x)\cdot m_{n}(x)\cdot m_i(x)^2=\lim_{N\to\infty}\frac1{[N/p]}\sum_{n=1}^{[N/p]}m_{n+k}(x)\cdot m_{n}(x),$$
for each $k,n\geq0$
\beq\label{ms6}
\widehat\sigma_x(kp^n)=\widehat{\sigma}_x(k).\eeq

\begin{Cor}\label{cms1}
Assume that $x=B\times B\times\ldots$ Let $r\neq s$. If there exists $t\geq1$ such that
$c_1:=\widehat\sigma(tr)\neq c_2:=\widehat\sigma(ts)$, then $\sigma_x^{(r)}\perp \sigma_x^{(s)}$.\end{Cor}
\begin{proof} It is well-known that $\sigma_f=\sigma_x$, e.g.\ \cite{Le}. The result then follows in view of~Corollary~\ref{cwo} (with $q_n=p^n$), (\ref{ms3}) and~(\ref{ms6}).\end{proof}

\begin{Remark}\label{r} \em We can slightly strengthen this result by observing that whenever $|c_1|\neq|c_2|$ then $\sigma_x^{(r)}\perp \sigma_x^{(s)}\ast\delta_{z_0}$ for each $z_0\in\bs^1$. Indeed, we only need to show that whenever $z^{q_n}\to c$ weakly in $L^2(\bs^1,\sigma)$ (that is, $\int_{{\bs^1}}z^{q_n+m}\,d\sigma(z)\to c\int_{\bs^1}z^m\,d\sigma(z)$ for each $m\in\Z$) and $z_0^{q_n}\to \alpha$, $|\alpha|=1$, then $z^{q_n}\to\alpha c$ weakly in $L^2(\bs^1,\sigma\ast\delta_{z_0})$.\end{Remark}

Consider now $m_{n}=(-1)^{s_2(n)}$, where $s_2(n)=0$ if the number of $1$s in the binary expansion of $n$ is even, and $1$ otherwise (cf.\ footnote~\ref{zielona} with $E=\N$). Directly from that we have $m_{2n}=m_n$ and $m_{2n+1}=-1$. Moreover, for each $n\geq0$
\beq\label{ms7}
m_n=m_n(01\times01\times\ldots)\eeq
The sequence $x=01\times01\times\ldots$ is called the Thue-Morse sequence. From now on, we deal only with $\sigma=\sigma_x$, where $x$ is the Morse-Thue sequence.

Since, for each $k\geq0$, $\widehat{\sigma}(k)=\lim_{N\to\infty}\frac1N\sum_{n=1}^Nm_{n+k}\cdot m_n$, we obtain
\beq\label{sm}
\widehat{\sigma}(2k)=\widehat{\sigma}(k),\;\;\widehat{\sigma}_{2k+1}=-\frac12\big(
\widehat{\sigma}(k)+\widehat{\sigma}(k+1)\big)~\footnote{Indeed, $m(1)=-1$, so
$$
\frac1N\sum_{n=1}^Nm_{n+2k+1}\cdot m_{n}=$$
$$
\frac1N\left(\sum_{n=1}^{[N/2]}m_{2n+2k+1}\cdot m_{2n}+\sum_{n=1}^{[N/2]}m_{2n+1+2k+1}\cdot m_{2n+1}+\mbox{O}(1)\right)=$$
$$
-\frac12\left(\frac1{[N/2]}\sum_{n=1}^{[N/2]}m_{n+k}\cdot m_{n}+\frac1{[N/2]}\sum_{n=1}^{[N/2]}m_{n+k+1}\cdot m_{n}\right)+\mbox{o}(1/N),$$
whence~(\ref{sm}) follows.}.\eeq
Moreover, we can check directly that $\widehat{\sigma}(0)=1$ and $\widehat{\sigma}(1)=-\frac13$.

\section{Fourier transform of the spectral measure of the Thue-Morse sequence. Spectral disjointness of powers}

Following (\ref{sm}), we
consider the sequence:
\beq\label{mt1}\begin{array}{lll} \widehat{\sigma}(0)=1,\widehat{\sigma}(1)=-\frac13,\widehat{\sigma}(2n)=\widehat{\sigma}(n),&&\\
\widehat{\sigma}(2n+1)=-\frac12(\widehat{\sigma}(n)+\widehat{\sigma}(n+1))&\mbox{for}&n\geq1.\end{array}\eeq

In what follows we need the values of $\widehat{\sigma}$ for
small odd numbers\footnote{We list some other values at prime
instances to see that there are equalities between many Fourier coefficients: $\widehat{\sigma}(17)=\widehat{\sigma}(31)=\frac1{12}$,
$\widehat{\sigma}(19)=\widehat{\sigma}(23)=\widehat{\sigma}(29)=-\frac1{12}$,
$\widehat{\sigma}(37)=\widehat{\sigma}(41)=\widehat{\sigma}(59)=-\frac1{24}$,
$\widehat{\sigma}(43)=\widehat{\sigma}(53)=\frac1{24}$,
$\widehat{\sigma}(47)=\widehat{\sigma}(61)=-\frac18$.}:
$$
\begin{array}{cccc}
\widehat{\sigma}(1)=-\frac13,& \widehat{\sigma}(3)=-\frac13,& \widehat{\sigma}(5)=0,&\widehat{\sigma}(7)=0,\\
\widehat{\sigma}(9)=-\frac16,&\widehat{\sigma}(11)=-\frac16,&\widehat{\sigma}(13)=-\frac16,&\widehat{\sigma}(15)=\frac16,
\end{array}
$$

\begin{Lemma}\label{rekurencja} For any natural numbers $n,a\geq1$, we have
$$
\widehat{\sigma}(2^an+1)=\left(-\frac12\right)^a
\left(\widehat{\sigma}(n+1)+\frac13\widehat{\sigma}(n)\right)-\frac13\widehat{\sigma}(n).
$$
\end{Lemma}
\begin{proof}
The assertion follows by induction on $a$.
\end{proof}

For a non-zero rational number $w=\frac{2^kp}{2^{\ell}q}$, where
$p,q\in\Z$ are odd, by $v_2(w)$ we denote the multiplicity of $2$
in $w$, i.e.\ $v_2(w)=k-l$.

\begin{Lemma}\label{st} If $v_2(w_1)> v_2(w_2)$  then $v_2(w_1+w_2)=v_2(w_2)$.\end{Lemma}
\begin{proof}
Writing $w_i=\frac{2^{k_i}p_i}{2^{l_i}q_i}$, $i=1,2$, we obtain
$$
w_1+w_2=\frac{2^{k_2+l_1}\big(2^{k_1+l_2-(k_2+l_1)}p_1q_2+p_2q_1\big)}{2^{l_1+l_2}q_1q_2}$$
since, by assumption, $k_1+l_2> k_2+l_1$, and the result
follows.\end{proof}

Given an odd natural number $K$, we define sequences $K_0,\ldots,K_r$
and $a_1,\ldots,a_r$ of natural numbers, where each $K_i$ is odd and:
$$
\begin{array}{l}
K_0=K,\;K_r=1,\\
K_{i-1}=2^{a_i}K_i+1
\end{array}
$$
for $i=1,\ldots,r$. Set $r(K):=r$ and $l(K):=a_1+\ldots+a_r$.

\begin{Lemma}\label{logarytm}
For an odd natural number $K$, we have $l(K)=[\log_2K]$; in particular
$$
2^{l(K)}< K< 2^{l(K)+1}
$$
\end{Lemma}
\begin{proof}
Follows by induction on $r(K)$.
\end{proof}

\begin{Lemma}\label{sigma}
If $K$ is an odd natural number then $\widehat{\sigma}(K)=0$ or
$v_2(\widehat{\sigma}(K))\ge 2-l(K)$. If  $K\ge 9$ then
$$
v_2(\widehat{\sigma}(K))=2-l(K).
$$
\end{Lemma}
\begin{proof}
We proceed by induction on $l(K)\ge 3$. If $l(K)=3$ then (see Lemma~\ref{logarytm})  $K\in\{9,11,13,15\}$ and we see that the assertion is true.

Assume that $l(K)>3$. Then $K=2^{a_1}K_1+1$. We consider the cases
$1\le K_1\le 7$ and $K_1>7$ separately.

If $K_1=1$ then $l(K)=a_1>3$ and (by Lemma~\ref{rekurencja})
$$
\widehat{\sigma}(K)=\big(-\frac12\big)^{a_1}\left(
\widehat{\sigma}(2)+\frac13\widehat{\sigma}(1)\right)-\frac13\widehat{\sigma}(1)
=\big(-\frac12\big)^{a_1}\big(-\frac49\big)+\frac19=\frac{-(-1)^{a_1}+2^{a_1-2}}{9\cdot
2^{a_1-2}}
$$
and $v_2(\widehat{\sigma}(K))=2-a_1=2-l(K)$.

If $K_1=3$ then $l(K)=a_1+1>3$ and
$$
\widehat{\sigma}(K)=\big(-\frac12\big)^{a_1}\left(
\widehat{\sigma}(4)+\frac13\widehat{\sigma}(3)\right)-\frac13\widehat{\sigma}(3)
=\big(-\frac12\big)^{a_1}\big(-\frac29\big)-\frac19=\frac{-(-1)^{a_1}-2^{a_1-1}}{9\cdot
2^{a_1-1}}
$$
and $v_2(\widehat{\sigma}(K))=1-a_1=2-l(K)$.

If $K_1=5$ then $l(K)=a_1+2> 3$ and
$$
\widehat{\sigma}(K)=\big(-\frac12\big)^{a_1}\left(
\widehat{\sigma}(6)+\frac13\widehat{\sigma}(5)\right)-\frac13\widehat{\sigma}(5)
=\big(-\frac12\big)^{a_1}\frac13=\frac{(-1)^{a_1}}{3\cdot 2^{a_1}}
$$
and $v_2(\widehat{\sigma}(K))=-a_1=2-l(K)$.

If $K_1=7$ then $l(K)=a_1+2> 3$ and
$$
\widehat{\sigma}(K)=\big(-\frac12\big)^{a_1}
\left(\widehat{\sigma}(8)+\frac13\widehat{\sigma}(7)\right)-\frac13\widehat{\sigma}(7)
=\big(-\frac12\big)^{a_1}\big(-\frac13\big)=\frac{(-1)^{a_1+1}}{3\cdot 2^{a_1}}
$$
and $v_2(\widehat{\sigma}(K))=-a_1=2-l(K)$.

Now consider the case $K_1>7$. Then $K_1+1=2(2^{a_2-1}K_2+1)=2^sL$
for some $s\ge 1$ and an odd number $L$. Observe that
$l(L)<l(K_1)$ (otherwise, $K_1<2^{m+1}$ and $2^m<L$ for some $m$
which leads to a contradiction: $2^{m+s}<2^sL=K_1+1\le 2^{m+1}$).
It follows that either
$\widehat{\sigma}(K_1+1)=\widehat{\sigma}(L)=0$  or
$v_2(\widehat{\sigma}(K_1+1))=v_2(\widehat{\sigma}(L))= 2-l(L)> 2-l(K_1)$ thanks to  the
inductive assumption. It follows that
$$v_2\left(\big(-\frac12\big)^{a_1}\widehat{\sigma}(K_1+1)\right)>2-a_1-l(K_1)=2-l(K),$$
and similarly
$$
v_2\left(\big(-\frac12\big)^{a_1}\frac13\widehat{\sigma}(K_1)\right)=2-a_1-l(K_1)=2-l(K)
$$
by the inductive assumption.  Moreover
$$v_2\left(\frac13\widehat{\sigma}(K_1)\right)=2-l(K_1)>2-l(K).$$
Recalling that
$$
\widehat{\sigma}(K)=
\big(-\frac12\big)^{a_1}\left(\widehat{\sigma}(K_1+1)+\frac13\widehat{\sigma}(K_1)\right)-
\frac13\widehat{\sigma}(K_1),
$$
using Lemma~\ref{st}, we finally obtain that
$v_2(\widehat{\sigma}(K))=2-l(K)$.
\end{proof}

Two numbers $K,L\geq 0$ are called {\em TM-equivalent} if $\widehat{\sigma}(2K+1)=\widehat{\sigma}(2L+1)$.


\begin{Lemma}\label{lmlsk} Assume that $K,L\ge 9$ are odd natural numbers. If $L$ is TM-equivalent to $K$ then $l(K)=l(L)$.
\end{Lemma}
\begin{proof}
Since $v_2(\widehat\sigma(K))=v_2(\widehat\sigma(L)))$ and $K,L\ge
9$ then, in view of Lemma~\ref{sigma}, we get $l(K)=l(L)$.
\end{proof}

Finally, we need the following observation.

\begin{Lemma}\label{lmt7}
Assume that $r,s\geq1$, $r<s$,  are odd numbers. Then for any
integer $a\geq1$ large enough there exist an odd number $t$  such
that
$$
rt<2^a<st.$$\end{Lemma}
\begin{proof}It is enough to consider $s=r+2$. Then, for each $a\geq1$ sufficiently large, we have
$$
r\left(\frac{2^a}{r+1}+2\right)<2^a<(r+2)\left(\frac{2^a}{r+1}-1\right),$$
so it is enough to take for $t$ either $\left[\frac{2^a}{r+1}\right]$ or $\left[\frac{2^a}{r+1}\right]+1$ (for $a$ sufficiently large).
\end{proof}

Using Lemmas~\ref{sigma},~\ref{lmt7},~\ref{lmlsk} and
Corollary~\ref{cms1} and noticing that
$\sigma^{(k)}=\sigma^{(2k)}$  for each $k\geq1$ (equivalently,
$\sigma$ is invariant under the map $z\mapsto z^2$),
 we obtain the following result.

\begin{Prop}\label{pspdisj} Let $\sigma$
denote the spectral measure associated to the
 Thue-Morse sequence $x=01\times01\times\ldots$
 Then $\sigma^{(r)}\perp\sigma^{(s)}$ if and only if
 $\max\{r/s,s/r\}\notin\{2^a:\:a\in\N\}$. Moreover, $\sigma^{(r)}=\sigma^{(2^ar)}$
 for each integers $a,r\geq0$.\end{Prop}

\begin{Cor}\label{cspdisj}
 Let $\sigma$ denote the spectral measure associated to the Thue-Morse
 sequence $x=01\times01\times\ldots$ Then $\sigma^{(p)}\perp\sigma^{(q)}$
 for arbitrary odd numbers $p\neq q$.\end{Cor}

\begin{Remark}\label{r1} \em Recall \cite{Qu} that the maximal spectral
 type  of the Thue-Morse dynamical system is given by the measure which is the sum of two measures: the discrete measure concentrated on all roots of unity of degree $2^n$, $n\geq0$, and the continuous measure $\widetilde{\sigma}$ which is the convolution of the discrete measure with $\sigma$. In view of Remark~\ref{r}, we obtain that the assertions of Proposition~\ref{pspdisj} and Corollary~\ref{cspdisj} are true when $\sigma$ is replaced by $\widetilde{\sigma}$.\end{Remark}

\section{On Sarnak's conjecture  for the Thue-Morse dynamical system}
As before, we let $x$ denote the Thue-Morse sequence. Set $X=\co(x)$
and notice that the map $\tau$ which interchange $0$s and $1$s is a
homeomorphism (involution) preserving $X$ and commuting with the shift $T$.
 Denote also by $x$ any extension of $x$ to a two-sided member of $X$. Take $f\in C(X)$ arbitrary and note that
\beq\label{sa1} f=f_t+f_m,\eeq where $f_t=\frac12(f+f\circ\tau)$
and  $f_m=\frac12(f-f\circ\tau)$. Then $f_t=f_t\circ\tau$,
$f_m\circ \tau=-f_m$.

Basic spectral theory for group extensions shows that
the spectral measure $\sigma_{f_m}$ of $f_m$ is absolutely
continuous with respect to $\widetilde\sigma$, so by Remark~\ref{r1} and an observation from \cite{Ab-Le-Ru},
$$
\frac1N\sum_{n=1}^Nf_m(T^ny)\mob(n)\to 0$$
for {\bf each} $y\in\co(x)$~\footnote{In particular,
it holds for $f(y)=(-1)^{y(0)}$ which shows that that
not only $x$ but each $y\in\co(x)$ is orthogonal to the M\"obius function.}.

It follows that to show Sarnak's conjecture for the
Thue-Morse dynamical system we need to prove it for
the associated Toeplitz dynamical system \cite{Le}. We recall that the Toeplitz sequence $z$, called from now on, Thue-Toeplitz, associated
with the Thue-Morse sequence $x$ arises as follows:
At first step we
we put $z(0)=z(2)=\ldots=z(2n)=\ldots=1$ and leave odd
places undefined. At the second step we
put $z(1)=z(5)=\ldots=z(4n+1)=\ldots=0$, that is, we fill
every second unfilled place by putting 0 here. At the
third step we repeat step one setting $1$ at every second
unfilled place, etc. It is not hard to see that for each $n\geq0$
\beq\label{ole}
z=B_n?B_n?B_n?\ldots,
\eeq
where $|B_n|=2^n-1$~\footnote{In fact, $z(i)=x(i)+x(i+1)$ mod~2.}, and ``?'' stands for un
unfilled place (half of these unfilled places will be filled at step $n+1$ of the construction).
To prove Sarnak's conjecture for $(\co(z),T)$, we need to consider a linearly dense set of functions in $C(\co(z))$.

Given a sequence $w=(w(i))_{i\in\Z}$ and $a\in\Z$, $l\ge 0$, we denote by $w[a,a+l)$ the sequence $(w(a),w(a+1),\ldots,w(a+l-1))$.
We consider maps defined as follows:
Fix $\ell\geq1$, $a\in\Z$ and consider
$f:\{0,1\}^{\ell}\to\C$. It
has a natural extension to a continuous map $f:\co(z)\to\C$ by
setting
$$
f(w)=f(w[a,a+\ell)),\;\;w\in\co(z).
$$
Under this notation, $f(T^kw)=f(w[a+k,a+k+\ell))$.

Observe that any continuous map on $\co(z)$ having only finitely many values is obtained this way.

Given  $\ell\geq1$, $a\in\Z$ and a sequence $w'\in\{0,1\}^{\ell}$, we denote by $U_{w',a}$ the open set
$\{w\in\co(z):w[a,a+\ell)=w'\}$.

 \begin{Lemma}\label{dense}
The functions $f:\co(z)\rightarrow \C$ taking finitely many values
form a dense subset in $C(\co(z))$.
 \end{Lemma}
 \begin{proof}
Fix $\varepsilon>0$ and let $f\in C(\co(z))$. Since $f$ is uniformly continuous, there exist $\ell\geq1$ and $a<0$ such that the diameter of $f(U_{w',a})$ is less than
$\varepsilon$ for any $w'\in\{0,1\}^{\ell}$. Fix such an $a\in\Z$. Let $w_1,\ldots,w_M\in\co(z)$ be the representatives of the
equivalence classes of the relation $\sim$ defined by $w\sim w'$ if and only if $w[a,a+\ell)=w'[a,a+\ell)$.
Let $f'$ be given by the formula
$f'(w)=f(w_j)$,
where $w[a,a+\ell)=w_j[a,a+\ell)$. Then $|f(w)-f'(w)|<\varepsilon$ for any $w\in\co(z)$.
\end{proof}

\begin{Lemma}\label{main} We have
$\frac{1}{N}\sum_{k=1}^Nf(T^kw)\mob(k)\rightarrow 0$ for any $f\in
C(\co(z))$ taking finitely many values and any $w\in\co(z)$.
\end{Lemma}
\begin{proof}

We can assume that there are $a\in\Z$, $\ell\in\N$ such that the value  $f(w)$ depends only on $w[a,a+\ell)$.

The map $f$ is bounded, say $|f(w)|\le F$ for all $w\in\co(z)$.

Fix $\varepsilon>0$ and let $n$ be a natural number such that
\beq\label{star1}\frac{\ell
F}{2^n}<\frac{\varepsilon}{3}.\eeq

We recall that $\frac1N\sum_{k=1}^Nb_k\mob(k)\to0$ for any
periodic (or eventually periodic with a fixed length of preperiod) sequence $(b_k)\subset \C$
(by the distribution of prime numbers in arithmetic progressions and the fact that the linear space of such sequences is of finite dimension).

Let $M_1$ be a number such that
\beq\label{star2}\frac1N\sum_{k=1}^Nb_k\mob(k)<\frac{\varepsilon}{3\cdot2^n}\eeq
for any periodic sequence $(b_k)$  bounded by $F$ of period at
most $2^n$ and any $N\ge 2^nM_1$. Assume moreover that
$\frac{F}{M_1}<\frac{\varepsilon}{3}$.

Take $N>2^nM_1$, put $M=[N/2^n]$ and let $w\in\co(z)$.  Then $w[1,N+1)=z[i,i+N)$ for
some $i\in \Z$, thus (see~(\ref{ole}))
$$
w[1,N+1)=C'B_ny_1B_ny_2\ldots B_ny_MC'',
$$
where $B_n$ is a $\{0,1\}$-sequence of length $2^n-1$,  $C',C''$
are $\{0,1\}$-sequences such that $C'$ is a suffix of $B_ny_0$, $C''$ is a prefix of $B_n$ and
$y_1,\ldots,y_M\in\{0,1\}$. Let $c',c''$ denote the length of
$C',C''$, respectively.

Then $\sum_{k=1}^Nf(T^kw)\mob(k)$ equals
{\scriptsize $$
\begin{array}{l}
f(Tw)\mob(1)+f(T^2w)\mob(2)+\ldots+f(T^{c'}w)\mob(c')+\\
f(T^{c'+1}w)\mob(c'+1)+f(T^{c'+2^n+1}w)\mob(c'+2^n+1)+\ldots+f(T^{c'+(M-1)2^n+1}w)\mob(c'+(M-1)2^n+1)+\\
f(T^{c'+2}w)\mob(c'+2)+f(T^{c'+2^n+2}w)\mob(c'+2^n+2)+\ldots+f(T^{c'+(M-1)2^n+2}w)\mob(c'+(M-1)2^n+2)+\\
......\\
f(T^{c'+2^n}w)\mob(c'+2^n)+f(T^{c'+2\cdot2^n}w)\mob(c'+2\cdot2^n)+\ldots+f(T^{c'+M2^n}w)\mob(c'+M2^n)+\\
f(T^{c'+M2^n+1}w)\mob(c'+M2^n+1)+f(T^{c'+M2^n+2}w)\mob(c'+M2^n+2)+\ldots+f(T^{c'+M2^n+c''}w)\mob(c'+M2^n+c'')
\end{array}
$$}
Having organized the sum into $2^n+2$ rows as above, we denote by
$E'$ the sum in the first row, by $E''$ the sum in the last one
and by $\Sigma_i$ the sum in the $i+1$-st row for
$i=1,\ldots,2^n$.

Observe that \beq\label{star3} |E'+E''|\le 2^nF. \eeq
 Moreover, $\Sigma_i$ is an
expression of the form $\Sigma_{k=1}^Nb_k\mob(k)$, where $(b_k)$ is
a periodic sequence,\footnote{The periodic sequence $(b_k)$ has
form $(0,...,0,\phi,0,...,0,\phi,0,...)$ and period $2^n$, where
$\phi$ is a fixed value of $f$.} provided the segment
$w[c'+i,c'+i+\ell)$ does not meet any "uncertain" entry $y_j$. The
latest condition holds for $i\le 2^n-\ell$.

It follows by (\ref{star2}) that
$$
|\Sigma_i|\le \frac{N\varepsilon}{3\cdot2^n}
$$
for $i=1,\ldots,2^n-\ell$.

For the remaining $i $'s we have \beq\label{star4} |\Sigma_i|\le
MF. \eeq Finally, taking into account (\ref{star1}), (\ref{star3}),
(\ref{star4}) and $\frac{F}{M_1}<\frac{\varepsilon}{3}$, we get
 $$
\left|\frac1N
\sum_{k=1}^Nf(T^kw)\mob(k)\right|<\frac{2^nF}{N}+\frac{2^n\varepsilon}{3\cdot2^n}+\frac{MF}{N}\le
\varepsilon.
$$
\end{proof}

\begin{Prop}\label{TS}
$\frac1N\sum_{k=1}^Nf(T^kw)\mob(k)\rightarrow 0$ for any $f\in
C(\co(z))$ and any $w\in\co(z)$.
\end{Prop}
\begin{proof} Follows  by Lemma~\ref{main} and Lemma~\ref{dense} and the fact that $\frac1N\sum_{k=1}^N\mob(k)\to0$.
\end{proof}

\section{On Sarnak's conjecture  for  Toeplitz dynamical systems}\label{ostatnia}
We will consider sequences of a finite alphabet $A=\{0,\ldots,d-1\}$, for some $d\geq2$. Following \cite{Ja-Ke}, a sequence $z\in A^{\N}$ is called a Toeplitz sequence if
for each $n\geq 0$ there is $a_n\geq1$ such that  $z(n)=z(n+a_n)=z(n+2a_n)=\ldots$ It can be proved \cite{Ja-Ke}, \cite{Wi} that then there is an increasing sequence $(p_n)$, $p_n|p_{n+1}$ such that for each $n\geq1$, $z$ can be represented as
$z=C_nC_n\ldots$ with $C_n$ being a block over the alphabet $A\cup\{?\}$, where the sign  ``?'' means an ``unfilled''  place (at the $n$th step of the construction of $z$) and $|C_n|=p_n$, $n\geq1$. Recall that a Toeplitz sequence $z$ is called {\em regular} if
$$
(\mbox{the number of unfilled places in}\;C_n)/p_n\to 0\;\;\mbox{when}\; n\to\infty.$$
The dynamical system generated by a regular Toeplitz sequence is uniquely ergodic and has zero entropy \cite{Ja-Ke}, \cite{Wi} (see also \cite{Do}). Moreover, dynamical systems generated by
Toeplitz sequences are so called almost 1-1 extensions of their maximal equicontinuous factors which are $(p_n)$-odometers.

Now, the same method which has been used to show that Sarnak's conjecture holds for the Thue-Toeplitz sequence shows the following.

\begin{Prop}\label{ts} Let $z$ be any regular Toeplitz sequence. Then the dynamical systems $(\co(z),T)$ determined by $z$ satisfies~(\ref{s1}).\end{Prop}

We will now show however that if the regularity assumption is dropped, Sarnak's conjecture may fail as the following construction shows.

Let $(a_n)_{n\in\N}$ be an increasing sequence of natural numbers such that $a_n|a_{n+1}$ and
\beq\label{an}
\rho:=\sum\limits_{n=1}^{\infty}\frac{1}{a_n}\le \frac14.
\eeq
(For instance, $a_n=5^n$.)

Let us denote $\N_0=\{0,1,2,...\}$.  We construct a Toeplitz
sequence $z=(z(n))_{n\in\N_0}$ (with terms 0, $\pm1$) as follows:
at the first stage of the construction we put $\mob(0):=0$ at each
$a_1$th place of $z$, starting form 0. At the second stage we find
the first unfilled place (in fact, it is the place number 1) and
then we put $\mob(1)$ at each $a_{2}$th of the unfilled places. We
proceed this way inductively: at the $n$th stage we find the first
unfilled place, say $m$, and we put $\mob(m)$ at each $a_{n+1}$th
unfilled place.

More precisely,
we define by induction subsets $A_n\subset \N_0$, $n\in\N_0$ with
the properties:
\begin{enumerate}
\item $n\in A_n$,
\item if $A_n\neq A_m$ then $A_n\cap A_m=\emptyset$,
\item $A_n$ is an arithmetic progression $\{m,m+a_{m+1},m+2a_{m+1},...,m+ka_{m+1},...\}$, where $m$ is the smallest element of $A_n$.
\end{enumerate}

We put $A_0=\{ka_1:k\in\N_0\}$ .

Assume that $n>0$ and $A_0,...,A_{n-1}$ are defined and satisfy
the above conditions. If $n\in A_m$ for some $m<n$, then we put
$A_n=A_m$ (by the property 2 the definition does not depend on
$m$). If $n\notin A_0\cup...\cup A_{n-1}$ we put
$A_n=\{n+ka_{n+1}:k\in\N_0\}$. We need to show that $A_n\cap
A_m=\emptyset$ for $m<n$ in the latter case. Indeed, assume that
$n+ka_{n+1}\in A_m$ for some $k\in\N_0$, where $m<n$. Assume that
$m$ is the smallest number with that property.  Then, by the
property 3,  $n+ka_{n+1}=m+la_{m+1}$ for some $l\in\N_0$ and
$n=m+(l-k\frac{a_{n+1}}{a_{m+1}})a_{m+1}$ hence $n\in A_m$.

We call the smallest element of $A_m$ the {\em initial of} $A_m$.
By an {\em initial} we call the initial of some $A_m$.


Now we define the sequence $z=(z(n))_{n\in\N_0}$ by $z(n):=\mob(m)$,
where $m$ is the initial of $A_n$. Thanks to the property~3, $z$
is a Toeplitz sequence.



\begin{Prop}
$$
\liminf\limits_{N\rightarrow\infty}\frac{1}{N}\sum\limits_{k=1}^Nz(k)\mob(k)\ge
\frac{6}{\pi^2}-2\rho>0.
$$
\end{Prop}

\begin{proof}

If $k$ is an initial then $z(k)=\mob(k)$. We find a bound of the
number of non-initials $k\in[1,N]$. If $k\in[1,N]$ is not an
initial, then $k\in A_k=A_m$, where  $m<k\le N$, $m$ is the
initial of $A_k$  and $a_{m+1}<N$. Clearly, $|[1,N]\cap \big(A_m\setminus\{m\}\big)|\le
\frac{N}{a_{m+1}}$, since $A_m$ is an arithmetic progression with
the difference $a_{m+1}$. Thus the number of non-initials in
$[1,N]$ is less than
$$
N(\frac{1}{a_1}+...+\frac{1}{a_{k}}+...)<N\rho.\footnote{The frequency of non-initials in $[1,N]$ does not tend to 1 - this is related with the fact that $z$ is not regular.}
$$

Since $|z(n)\mob(n)|\le 1$ for $n\in\N$ and $z(k)=\mob(k)$ for any
initial  $k$, it follows that \beq\label{mu2}
\sum\limits_{k=1}^Nz(k)\mob(k)\ge \sum\limits_{k=1}^N\mob(k)^2 -
2N\rho. \eeq


We have $\mob(k)^2=1$ for square-free $k$, hence by~(\ref{mu2}),  we
obtain \beq \sum\limits_{k=1}^Nz(k)\mob(k)\ge |\{k\in[1,N]:
k\;\;\mbox{\rm square-free}\}|-2N\rho. \eeq

The frequency of square-free numbers in $[1,N]$ tends to
$\frac{6}{\pi^2}$, thus
$$
\liminf\limits_{N\rightarrow\infty}\frac{1}{N}\sum\limits_{k=1}^Nz(k)\mob(k)\ge
\frac{6}{\pi^2}-2\rho
$$
and $\frac{6}{\pi^2}-2\rho>0$ by (\ref{an}).

\end{proof}

\begin{Remark}\em It follows that
Sarnak's conjecture does not hold for the topological dynamical system $(\co(z),T)$  associated to $z$. It can also be proved that the topological entropy of the associated dynamical system is positive.

\end{Remark}

\section{Generalized Morse sequences having a stabilizing Thue-Morse subsequence}\label{ostatnia1} We now come back to the cocycle representation of Morse dynamical systems, to see how to obtain more systems for which Sarnak's conjecture holds.

Let $y=b^0\times b^1\times\ldots $  be a (generalized) Morse sequence.
Fix $s\geq 1$ and then $k,K\geq1$ (we think of $K$ as going to $\infty$) and write
\beq\label{mor1}
y=B\times C\times b^{k+K}\times b^{k+K+1}\times\ldots, \eeq
where $B=b^0\times\ldots\times b^{k-1}$, $C=b^k\times\ldots\times b^{k+K-1}$.
Notice that this operation of putting parentheses does not change the space $\co(y)$, in particular, the number (cf.\ (\ref{ms2}))
\beq\label{mor2}
\int_{\co(y)}(-1)^{u(s|B|)-u(0)}\,d\mu_y(u)=
\int_{\co(y)}(-1)^{u(sq_k)-u(0)}\,d\mu_y(u)=\widehat{\sigma}_y(sq_k)\eeq
is not changed compared to the original representation of $y$. However, introducing parentheses does change the cocycle representation of the corresponding dynamical system (the underlying odometer has a different algebraic representation, and the Morse cocycle also changes). We now provide some details (see \cite{Le}) concerning the first two steps in the definition of the odometer and the Morse cocycle.

To the length $|B|$ of $B:=(b_0,b_1,\ldots,b_{|B|-1})$ there corresponds a partition $\cd^{(1)}=\{D^{(1)}_0,\ldots, D^{(1)}_{|B|-1}\}$ for which $SD^{(1)}_i=D^{(1)}_{i+1}$ modulo~$|B|$. At this stage, the Morse cocycle $\phi:Y\to\{0,1\}$ is defined as being constant on each $D^{(1)}_i$, $i=0,1,\ldots,|B|-2$ with
$$
\phi|_{D^{(1)}_{i}}=b_i+b_{i+1}\;\;(\mbox{modulo}~2),\;i=0,1,\ldots,|B|-2, $$
i.e.\ we set the successive values of the block $\check{B}:=(b_0+b_1)(b_1+b_2)\ldots(b_{|B|-2}+b_{|B|-1})$ as the values of $\phi$ on $D^{(1)}_{0},\ldots,D^{(1)}_{|B|-2}$,
and $\phi$ is not defined on $D^{(1)}_{|B|-1}$ (modulo~$2$). Also, denote $\Sigma:=\check{b}_0+\ldots+\check{b}_{|B|-2}$. At the next step, the tower $\cd^{(1)}$ is refined to a tower $\cd^{(2)}$ in the following way: We divide the base $D^{(1)}_0$ into $|C|$ equal pieces, which yields the partition of $\cd^{(1)}$ into columns and obtain a new tower $\cd^{(2)}$ (with $|B|\cdot|C|$ levels) in which
$$ SD^{(2)}_i=D^{(2)}_{i+1}\;\;\mbox{modulo}~|B|\cdot|C|\;\;\mbox{and}\;\; D^{(1)}_{|B|-1}=\bigcup_{j=1}^{|C|}D^{(2)}_{j|B|-1}.$$
We need to define $\phi$ on the levels of $\cd^{(2)}$ (a constant values on each level) on which $\phi$ is not defined yet, except of $D^{(2)}_{|C||B|-1}$ (the values of $\phi$ for $u\in D^{(2)}_{|C||B|-1}$ are settled in the following steps and depend successively on $b^{k+K}$, $b^{k+K+1}$, etc., this will not be relevant for our purposes). So for the top level $D^{(1)}_{|B|-1}=\bigcup_{j=1}^{|C|}D^{(2)}_{j|B|-1}$ we set $\check{C}+b_{|B|-1}$ as the consecutive values of $\phi$ (remembering that $b_0=0$), that is
\beq\label{mor3}
\phi|_{D^{(2)}_{j|B|-1}}=c_{j-1}+c_{j}+b_{|B|-1}\;\;\mbox{for}\;\;j=1,\ldots,|C|-1.\eeq
We now define a (partial) function $f_{s,C}$ in the following way:
$$
f_{s,C}|_{D^{(2)}_{j|B|-1}}=\check{c}_{j-1}+\check{c}_j+\ldots+\check{c}_{j+s-1}+sb_{|B|-1}+s\Sigma$$
for $j=1,\ldots, |C|-s$ and then we spread this constant value down the column for which $D^{(2)}_{j|B|-1}$ is the top level. It follows that the function $f_{s,C}$ is not defined on the last $s$ columns of the tower $\cd^{(1)}$, that is, on a set of measure $s/|C|$.  Moreover, the value
\beq\label{mor4}
\mbox{ $\big|\int_{Y}(-1)^{f_{s,C}(u)}\,d\nu(u)\big|$ (defined up to $\pm s/|C|$) does not depend on $B$}\eeq
and
\beq\label{mor5}
\left|\big| \int_{Y}(-1)^{\phi^{(s|B||C|)}(u)}\,d\nu(u)\big|-
\big|\int_{Y}(-1)^{f_{s,C}(u)}\,d\nu(u)\big|\right|<s/|C|.\eeq

Come back now to the Thue-Morse sequence $x=01\times01\times\ldots$, use~(\ref{ms2}) and~(\ref{ms3}) and apply the above to obtain
$$
\left|\big| \int_{Y}(-1)^{\phi^{(s\cdot 2^{k+K})}(u)}\,d\nu(u)\big|-
\big|\int_{Y}(-1)^{f_{s,C}(u)}\,d\nu(u)\big|\right|<s/2^K.$$
Taking into account that $\widehat{\sigma}(s)=\widehat{\sigma}(s\cdot 2^{\ell})$, as a conclusion of this, we obtain that for the Thue-Morse sequence we have
\beq\label{mor7} \left|
\int_{Y}(-1)^{f_{s,C}}\,d\nu\right|\tend{|C|}{\infty}|\widehat{\sigma}(s)|.\eeq

Now, we borrow the idea of stabilizing subsequence from \cite{Ab-Le-Ru}. We recall that a (generalized) Morse sequence $y=b^0\times b^1\times\ldots$ is of Thue-Morse type if it has a stabilizing Thue-Morse subsequence, therefore there exists a subsequence $k_1<k_2<\ldots$ such that for each $K\geq1$
$$
(b^{k_i},b^{k_i+1},\ldots,b^{k_i+K-1})=(01,01,\ldots,01)\;\;\mbox{eventually in $i$}.$$
Denote by $\eta$ the spectral measure of $y$. Using (\ref{ms3}), (\ref{mor4}), (\ref{mor5}) and~(\ref{mor7}) we obtain that for each $i\geq1$ and $K\geq1$ (sufficiently large)
$$
\big||\widehat{\eta}(sq_{Kk_i})|-|\widehat{\sigma}(s)|\big|<s/2^K.$$
Assume that $r\neq s$ are odd numbers. Then, in view of Lemma~\ref{lmt7}, without loss of generality we can assume that $|\widehat{\sigma}(r)|\neq|\widehat{\sigma}(s)|$. By the above, we can indicate easily a subsequence  $(L_i)$ such that
$$\left|\widehat{\eta}(sq_{L_i})|\to|\widehat{\sigma}(s)\right|\;\;\mbox{and}\;\;|
\widehat{\eta}(rq_{L_i})|\to|\widehat{\sigma}(r)|.$$
Applying Corollary~\ref{cwo}, we obtain $\widetilde{\eta}^{(s)}\perp\widetilde{\eta}^{(r)}$
and, moreover, Remark~\ref{r1} applies for $y$.
In view of Proposition~\ref{ts}, we have proved the following result.

\begin{Cor}\label{cs} Assume that $y=b^0\times b^1\times\ldots$ is a generalized Morse sequence of the Thue-Morse type. Then Sarnak's conjecture holds for the corresponding Morse dynamical system.\end{Cor}

\scriptsize

\noindent El Houcein El Abdalaoui:\\
Laboratoire de Math\'ematiques Rapha\"el Salem,\\
Universit\'e de Rouen, CNRS --
Avenue de l'Universit\'e --
76801 Saint \'Etienne du Rouvray, France\\
elhoucein.elabdalaoui@univ-rouen.fr

\vspace{2ex}

\noindent Stanis\l aw Kasjan, Mariusz Lema\'nczyk:\\ Faculty of Mathematics and Computer Science, Nicolaus Copernicus University, 12/18 Chopin street, 87-100 Toru\'{n}, Poland\\
skasjan@mat.umk.pl\\mlem@mat.umk.pl
\end{document}